\documentclass[fleqn,10pt,twocolumn]{SICE24}

\title{Low-rank approximated Kalman filter using Oja's principal component flow for discrete-time linear systems}

\author{Daiki Tsuzuki${}^{1\dagger}$ and Kentaro Ohki${}^{1}$}
\speaker{Daiki Tsuzuki}

\affils{${}^{1}$Department of Applied Mathematics and Physics, Graduate School of Informatics, Kyoto University, Kyoto, Japan\\
(Tel: +81-75-753-4754; E-mail: ohki@bode.amp.i.kyoto-u.ac.jp)\\
}
\abstract{%
The Kalman filter is indispensable for state estimation across diverse fields but faces computational challenges with higher dimensions. Approaches such as Riccati equation approximations aim to alleviate this complexity, yet ensuring properties like bounded errors remains challenging. Yamada and Ohki introduced low-rank Kalman-Bucy filters for continuous-time systems, ensuring bounded errors. This paper proposes a discrete-time counterpart of the low-rank filter and shows its system theoretic properties and conditions for bounded mean square error estimation. Numerical simulations show the effectiveness of the proposed method. 
}

\usepackage{color}
\allowdisplaybreaks


\begin{document}

\maketitle

%-----------------------------------------------------------------------

\section{Introduction}
% The Kalman filter \cite{kalman1960new} is a recursive optimal state estimation {\color{blue}algorithm} used in various fields such as weather forecasting \cite{angeo-24-2451-2006,HUANG1996159,Verlaan1998Efficinet}, finance \cite{wells2013kalman,harvey2009unobserved}, and power grid \cite{Cardoso2008,hoffmann2013minimal}. However, the computational complexity of the Kalman filter increases as the state and the output dimension increase. Many methods have been proposed to address this issue, including model reduction \cite{gugercin2004survey,Rowley2017} and approximations of the Riccati equation \cite{simon2006optimal,BENNER2016430,Simoncini2016,Bonnabel2013,schmidt2023rank}. Model reduction is mainly applicable to stable linear time-invariant systems and cannot preserve the physical meaning of the state due to state reduction. 
% On the other hand, {\color{blue}the numerical burden of the Kalman filter is mainly from the computation of the corresponding Riccati equation. Hence, it is a natural idea to slim the Riccati equation.  Besides,} 
% approximations of the Riccati equation are {\color{blue}easy to implement even for} nonlinear systems via the extended Kalman filter and can preserve the physical meaning of the state estimation. 
The Kalman filter \cite{kalman1960new} is a recursive optimal state estimation algorithm used in fields such as weather forecasting \cite{angeo-24-2451-2006,HUANG1996159,Verlaan1998Efficinet}, finance \cite{wells2013kalman,harvey2009unobserved}, and power grid management \cite{Cardoso2008,hoffmann2013minimal}. However, its computational complexity increases with the state and output dimensions. To address this issue, many methods like model reduction \cite{gugercin2004survey,Rowley2017} and Riccati equation approximations \cite{simon2006optimal,BENNER2016430,Simoncini2016,Bonnabel2013,schmidt2023rank} have been proposed. Model reduction is suitable for stable linear time-invariant systems but can compromise the physical meaning of the state. Conversely, the computational burden of the Kalman filter primarily stems from solving the Riccati equation, making it logical to simplify this equation. Approximations of the Riccati equation are straightforward to implement, even for nonlinear systems via the extended Kalman filter, and they preserve the physical meaning of state estimation.

%However, approximating the Riccati equation generally makes it difficult to guarantee desired properties such as the boundedness of estimation errors. Even when such properties can be guaranteed, identifying the conditions under which they hold is challenging. For continuous-time systems, Yamada and Ohki \cite{yamada2021new,yamada2021comparison} proposed low-rank Kalman-Bucy filters, which are modifications of the filters introduced in \cite{Bonnabel2013}. These filters guarantee bounded mean square estimation errors for linear time-invariant systems under certain rank conditions \cite{tsuzuki2024low}.
Approximating the Riccati equation, however, often makes it difficult to guarantee properties like bounded estimation errors, and identifying the conditions for these properties is challenging. For continuous-time systems, Yamada and Ohki \cite{yamada2021new,yamada2021comparison} proposed low-rank Kalman-Bucy filters, modified from \cite{Bonnabel2013}, ensuring bounded mean square estimation errors under specific rank conditions \cite{tsuzuki2024low}.

The low-rank filter integrates the Oja flow \cite{OJA198569,yan1994global} with a low-dimensional Riccati equation. The Oja flow can capture the largest real part of the eigenvalues of a square matrix \cite{tsuzuki2024low}, effectively identifying the unstable eigenvalues and their modes in linear continuous-time systems. However, this property is not directly applicable to linear discrete-time systems, where the stable eigenvalue region is inside the unit circle in the complex plane. Consequently, developing low-rank Kalman filters for discrete-time systems that ensure bounded estimation errors remains challenging.

%This paper considers a continuous-time system with discrete time observations as in \cite{schmidt2023rank} and develops a low-rank Kalman filter. Note that although this setting does not include general discrete-time systems, most physical systems are described by continuous-time dynamics; therefore, many practically meaningful systems can be treated in our approach. 

This paper addresses continuous-time systems with discrete-time observations, following \cite{schmidt2023rank}, and develops a low-rank Kalman filter. Although this covers only a part of discrete-time systems, it applies to many practical systems described by continuous-time dynamics.

The contributions of this paper are: 
(1) A new low-rank Kalman filter for a discrete-time system is proposed (Sec. \ref{sec:proposed_filter}). (2) The conditions for bounded mean square estimation errors for the proposed low-rank filter are derived (Theorem \ref{therem:stabilitycondtion_rank}). (3) The exact computational complexity of the proposed filter is shown (Table \ref{table:per_step_cal}).
% \begin{enumerate}
%     \item  A new low-rank Kalman filter for a discrete-time system
%            is proposed in Sec. \ref{sec:proposed_filter}.  
           
%     \item  The conditions for bounded mean square estimation errors for the proposed low-rank filter are derived in Theorem \ref{therem:stabilitycondtion_rank}.  
    
%     \item The exact computational complexity of the proposed filter is shown in Table \ref{table:per_step_cal}.
%     %The proposed low-rank filter reduces computational complexity compared to the Kalman filter.
% \end{enumerate}

The remainder of this paper is organized as follows. We briefly
 review the Kalman filter and Oja flow in Section \ref{sec:previous_work}.
 We propose a new low-rank Kalman filter and analyze the stability property of the filter in Section 3. In Section 4, we compare the Kalman filter to the proposed low-rank filter. We conclude this paper in Section 5.
 
{\bf{Notation}}
The sets of real and complex numbers are denoted by $\mathbb{R}$ and $\mathbb{C}$, respectively. The sets of $n \times m $ real and complex matrices are described by ${\mathbb{R}}^{n\times m}$ and ${\mathbb{C}}^{n \times m}$, respectively. The $n\times n$ identity matrix is denoted by $I_{n}$ and $n\times m$ zero matrix is denoted by $O_{n,m}$. For $n\times m$ matrix $A$, $A^{\top}$ and $A^{\dagger}$ are the transposed and Hermitian conjugates of $A$, respectively.
For a real symmetric matrix $A$, $A>0(\geq0)$ indicates that $A$ is positive (semi-)definite. 
For a positive semidefinite matrix $A\in {\mathbb{R}}^{n\times n}$, $\sqrt{A}$ indicates a positive semidefinite matrix $B \in {\mathbb{R}}^{n\times n}$ such that $A=B^{2}$.  
For a square matrix $A \in {\mathbb{C}}^{n \times n}$, its eigenvalues are denoted by $\lambda_{i}(A)\quad ({\mathrm{Re}}(\lambda_1(A)) \ge \dots \ge {\mathrm{Re}}(\lambda_n(A)))$, and the corresponding eigenvectors with norm 1 are denoted by $\bm{\psi}_{i}(A)$, $i=1,\dots, n$. %If the argument is clear, we simply write $\lambda _i$ and $\bm{\psi}_i$. 
For the degenerated eigenvalues, we use $\psi _{i}$ as generalized eigenvectors. 
The Stiefel manifold is denoted as ${\mathrm{St}}(r,n) \coloneqq \{X \in {{\mathbb{R}}^{n \times r}}|{X}^{\top}X=I_r\}$.

\section{previous work}
\label{sec:previous_work}
\subsection{Kalman filter}
In this study, we consider the following continuous-time system with discrete-time measurement:
\begin{align}
 \frac{dx(t)}{dt}&=Ax(t) + Gw(t), \label{eq.1} \\
y(t _{k})&=Cx(t_{k}) + H v(t_{k}), \label{eq.2}  
\end{align}
where $x(t) \in {\mathbb{R}}^{n}$ is the state variable, $w(t) \in {\mathbb{R}}^{q}$ is the process noise, $y(t_{k}) \in {\mathbb{R}}^{p}$ is the observed value at time $t_{k} := k h$ with a sampling period $h>0$ and a nonnegative integer $k\geq 0$, and $v(t_{k}) \in {\mathbb{R}}^{p}$ is the observation noise. 
$w(t)$ and $v(t_{k})$ are Gaussian noises that satisfy the following conditions: ${\mathbb{E}}[w(t)]=0$, ${\mathbb{E}}[v(t_{k})]=0$, ${\mathbb{E}}[w(t)w(\tau)^{\top}] = \delta(t-\tau)I_q$, ${\mathbb{E}}[v(t_{k})v(t_{\ell})^{\top}] = \delta _{k,\ell} I_p$, and ${\mathbb{E}}[v(t_{k})w(t)^{\top}] = O_{p,q}$ for any $t>0$ and $k$. We also assume ${\mathrm{det}}(H)\neq 0$.  
The system matrices $A$, $G$, $C$, and $H$ are appropriately sized matrices. %We impose a standard assumption; ${\mathrm{det}}(H)\neq 0$.
Throughout this paper, we also assume that there exists $r \in \{ 1,\dots ,n-1\}$ such that ${\mathrm{Re}}(\lambda _{r}(A)) > {\mathrm{Re}}(\lambda _{r+1}(A))$.  

Applying the lifting technique to Eq. \eqref{eq.1} gives the following equation. 
\begin{align*}
x(t_{k+1}) = & e^{Ah} x(t_k) + \int _{t_k}^{t_{k+1}} e^{A(t_{k+1} - \tau)}G w(\tau)d\tau .
\end{align*}
Defining $x[k]:= x(t_k)$, $y[k] := y(t_k)$, and $v_{d}[k]:=v(t_k)$, the discrete-time system is given as follows. 
% Additionally, we consider the following discrete-time system obtained by discretizing equations (\ref{eq.1}) and (\ref{eq.2}) with a sampling interval $h$.
\begin{align}
x[k+1]&=A_d x[k]+G_dw_d[k], \label{eq.3} \\
y[k]&=Cx[k] + H v_d[k], \label{eq.4} \\
% x[k]&=x(kh),\ A_d=e^{Ah},\notag\\
A_d&:=e^{Ah},\ G_d:=\sqrt{\int^{h}_{0} e^{A\tau}GG^{\top}e^{A^{\top}\tau}d\tau}\notag , 
\end{align}   
%where $k \geq 0$ is an integer. 
where $G_d$ is the normalization constant so that $w_d[k] \in {\mathbb{R}}^{n}$ follows mutually independent standard Gaussian distribution. 
Notice that while $w(t)$ is an $q$-dimensional vector, $w_d [k]$ is an $n$-dimensional vector so that $x_{d}[k]$ and $x(t_k)$ has the same statistical moments at each $t_{k}$.  
%We assume $M\coloneqq H_d{H_d}^{\top}>0$.

Let ${\mathcal{Y}}_k$ be the $\sigma$-field generated by measurement outcomes up to $t_{k}$, ${\mathcal{Y}}_k := \sigma \{ y[1],\dots,y[k] \}$, $\hat{x}_{k|\ell } = {\mathbb{E}}[x[k]|{\mathcal{Y}}_{\ell}]$, and  $\hat{P}_{k|\ell} = {\mathbb{E}}[(x[k]-\hat{x}_{k|\ell})(x[k]-\hat{x}_{k|\ell})^{\top}]$. Then, the Kalman filter is computed by the following steps: %the minimum variance estimation $ \hat{x}_{k|k}$ can be computed by the following steps:
\begin{description}
    \item[\upshape Step 1:] Initial conditions:
    \begin{align}
    \hat{x}_{0|-1} &= \bar{x}_0, \quad \hat{P}_{0|-1} = \Sigma_0 \label{eq:KFini}
    \end{align}
    
    \item[\upshape Step 2:] Kalman gain:
    \begin{align}
    \hat{K}_k &= \hat{P}_{k|k-1}C^{\top}(C\hat{P}_{k|k-1}C^{\top}+M)^{-1} , \label{eq:KFgain}
    \end{align}
    where $M\coloneqq H {H}^{\top}$.
    
    \item[\upshape Step 3:] Filter equations:
    \begin{align}
    \hat{x}_{k|k} &= \hat{x}_{k|k-1} + \hat{K}_k (y[k]-C\hat{x}_{k|k-1}) \label{eq:KFxs}\\
    \hat{x}_{k+1|k} &= A_d\hat{x}_{k|k} \label{eq:KFxs+1} 
    \end{align}
    
    \item[\upshape Step 4:] Error covariance matrix:
    \begin{align}
    \hat{P}_{k+1|k} &= A_d(I_n-\hat{K}_kC)\hat{P}_{k|k-1}A_d^{\top} + G_{d}^{2} \label{eq:Ps+1} 
    \end{align}   
\end{description}
The following proposition provides the steady-state solution to \eqref{eq:Ps+1}.
\begin{proposition}[{\cite[Sec. 7.3]{simon2006optimal}}]\label{prop:KF stability}
If the system $(A_d,G_d,C)$ is reachable and observable, then the algebraic Riccati equation:
\begin{align}
\hat{P}=A_d(I_n-\hat{K}C)\hat{P}A_d^{\top}+ G_d^{2} \notag
\end{align}
has a unique positive definite solution $\hat{P}>0$, and $A_d(I_n-\hat{K}C)$ is Schur stable. Furthermore, the solution $\hat{P}_{k+1|k}$ of Eq. \eqref{eq:Ps+1} converges to $\hat{P}$. Here, $\hat{K} = \hat{P}C^{\top}(C\hat{P}C^{\top}+M)^{-1}$. 
\end{proposition}

Note that a matrix $B \in {\mathbb{R}}^{n\times n}$ is Schur stable if $|\lambda_i(B)|<1$ for all $i=1,\dots,n$.

\subsection{Oja flow}
In this section, we introduce the differential equation called Oja flow that composes the proposal low-rank Kalman filter. Consider the following differential equation: for $U(0) \in {\mathrm{St}}(r,n)$,
%Oja flow has important properties like capturing the information of principal eigenvalues \cite{tsuzuki2024low}.   
\begin{align}
\varepsilon \frac{dU(t)}{dt} &= (I_n - U(t){U(t)}^\top)AU(t) , \label{eq:DFOja}%\notag   
\end{align}
 where $\varepsilon\in(0,1]$ is a small parameter to adjust the convergence speed of the above equation. 
 Equation \eqref{eq:DFOja} is known as the Oja flow \cite{OJA198569} and usually $\varepsilon=1$ is chosen. 
 Provided that it follows an Oja flow, $U(t)\in{\mathrm{St}}(r,n)$ holds at any time \cite{tsuzuki2024low}. 
 Let ${\mathbb{S}}$ be the set of equilibrium points for Oja flow. 
Note that if $\bar{U} \in {\mathbb{S}}$, then for any orthogonal matrix $W \in {\mathbb{R}}^{r\times r}$, $\bar{U}W \in {\mathbb{S}}$.  This means that equilibrium points can be an uncountable set, and the stability notion of each equilibrium point should be considered as follows: a set $\mathcal{U} \subset {\mathbb{S}}$ is said to be {\em asymptotically stable} if the perturbed orbit from any $\bar{U} \in \mathcal{U}$ returns to the set $\mathcal{U}$ and there exists a perturbed orbit that converges to $\bar{U}$ itself. We will refer to the following Propositions.  
%Oja flow captures the information of principal eigenvalues.
%Then, the following proposition provides the stability of Oja flow. 

\begin{proposition}[Adapted from {\cite[Thm. 1]{tsuzuki2024low}}]\label{prop:Ojaflow_stability}
Let $\mathcal{U}$ be $\{ (\bm{\psi}_1,\cdots,\bm{\psi}_r)$ $K \in {\mathrm{St}}(r,n) | K\in{\mathbb{C}}^{r\times r} \}$, where $\bm{\psi} _{i} = \bm{\psi} _{i}(A)$. 
%If ${\mathrm{Re}}(\lambda_r(A))>{\mathrm{Re}}(\lambda_{r+1}(A))$, 
Then, the set $\mathcal{U}$ is locally asymptotically stable, and the set ${\mathbb{S}}\setminus {\mathcal{U}}$ is unstable.
\end{proposition}

Proposition \ref{prop:Ojaflow_stability} indicates that the range of $\bar{U}\in{\mathcal{U}}$ is the linear subspace of the eigenvalues of the $r$-dominant eigenvalues of $A$. Furthermore, the $\mathcal{U}$ is the only stable equilibrium subset. Therefore, it is enough to consider an element of $\mathcal{U}$ for steady-state analysis, particularly for stability analysis of the proposed filter below. 
%Proposition \ref{prop:Ojaflow_stability} shows that $\bar{U}\in{\mathcal{U}}$ includes the information of principal eigenvalues. There are several disjoint equilibrium sets; however, ${\mathcal{U}}$ is the unique stable set. 
%See \cite{tsuzuki2024low} for the details. 
% \begin{remark}\label{rem.1}
% The properties such as the domain of attraction of the
% stable equilibrium set $\mathcal{U}$ of the Oja flow are detailed in \cite{tsuzuki2024low}.
% \end{remark}
We also introduce the following proposition. %Additionally, the following property holds.  

\begin{proposition}[{\cite[Prop. 3]{tsuzuki2024low}}]\label{prop.eigenvalue_preservation}
Let $\mathcal{U}$ be $\{(\bm{\psi}_1,\cdots,\bm{\psi}_r)$ $K \in {\mathrm{St}}(r,n) | K \in {\mathbb{C}}^{r\times r}\}$, where $\bm{\psi} _{i} = \bm{\psi} _{i}(A)\in {\mathbb{C}}^{n}$, and $U \in {\mathcal{U}}$. 
Then, $\lambda _{i} (A) = \lambda _{i}(U^{\top} A U)$ for $i=1,\dots ,r$.  
\end{proposition}

From Proposition \ref{prop.eigenvalue_preservation}, if ${\mathrm{Re}}(\lambda_{r+1}(A))<0$, then $\bar{U}$ cover the all unstable eigenvalue of $A$ . This property is related to how to cover the unstable eigenvalue of $A_d$ and the stability of the new low-rank Kalman filter.
%{\color{blue}For the discrete-time system matrix $A_d$, the ``worst" eigenvalue is $\lambda _{1}^{\rm d} := e^{\lambda _{1}(A)h}$.}

\section{main results}
\subsection{Low-rank Kalman filter for discrete-time systems}\label{sec:proposed_filter}
It is difficult to implement the Kalman filter if the system and output dimensions are large. To deal with this problem, we approximate $\hat{P}_{k|k-1}$ using a low-dimensional matrix $\tilde{R}_{k|k-1} \in {\mathbb{R}}^{r \times r}$ and $U_k \in {\mathrm{St}}(r,n)$. $\hat{P}_{k|k-1}$ is computed as $\tilde{P}_{k|k-1} = U_k \tilde{R}_{k|k-1} U_{k}^\top$, where $U_k := U(t_{k})$ is the solution of \eqref{eq:DFOja}. %the following equation:
% \begin{align}
% \varepsilon \frac{dU(t)}{dt} &= (I_n - U(t){U(t)}^\top)AU(t), \ t\in (t_{k}, t_{k+1} ] \label{eq:DFOja}\\ 
% U_k &= U(t_k)\notag.   
% \end{align}
Next, we introduce a low-rank Kalman filter algorithm corresponding to Eqs. (\ref{eq:KFini})-(\ref{eq:Ps+1}). 
%Next, we present the equations corresponding to (\ref{eq:KFini}) through (\ref{eq:Ps+1}). 
We can compute the estimated value $\tilde{x}_{k|k}$ at time $k$ using the following steps:
\begin{description}
    \item[\upshape Step 1:] Initial conditions:
    \begin{align}
    \tilde{x}_{0|-1} = \bar{x}_0, \quad \tilde{R}_{0|-1} = \tilde{\Sigma}_0 > 0 ,\quad U_{0} \in {\mathrm{St}}(r,n) \label{eq:LKFini}
    \end{align}

    \item[\upshape Step 2:] Solve \eqref{eq:DFOja} over $[t_{k-1},t_{k}]$ and obtain $U_{k}\in {\mathrm{St}}(r,n)$. 
    
    \item[\upshape Step 3:] The low-rank Kalman gain $\tilde{F}_k \in {\mathbb{R}}^{r \times p}$:
    \begin{align}
    \tilde{F}_k = \tilde{R}_{k|k-1} C_{U_k}^\top (C_{U_k} \tilde{R}_{k|k-1} C_{U_k}^\top + M)^{-1} \label{eq:LKFgain}
    \end{align}
    where $C_{U_k} \coloneqq CU_k$.
    
    \item[\upshape Step 4:] Low-rank filter equations:
    \begin{align}
     \tilde{x}_{k|k} &= \tilde{x}_{k|k-1} + U_k\tilde{F}_k \left(y[k] - C \tilde{x}_{k|k-1}\right) \label{eq:LKFxs}\\
     \tilde{x}_{k+1|k} &= A_d \tilde{x}_{k|k} \label{eq:LKFxs+1} 
    \end{align}
    
    \item[\upshape Step 5:] Low dimensional Riccati difference equation; for $\tilde{R}_{k+1|k} \in {\mathbb{R}}^{r \times r}$:
    \begin{align}
    \tilde{R}_{k+1|k} &= A_{U_k}(I_r-\tilde{F}_k C_{U_k}) \tilde{R}_{k|k-1}A_{U_k}^\top + G_{U_k} G_{U_{k}}^\top , \label{eq:Rs+1} 
    \end{align}
    where $A_{U_k} \coloneqq U_{k}^\top A_d U_k \in {\mathbb{R}}^{r\times r}$, and $G_{U_k} \coloneqq U_{k}^\top G_d \in {\mathbb{R}}^{r\times n}$.
\end{description}

An advantage of the proposed low-rank filter is that the calculation of the inverse matrix in \eqref{eq:LKFgain} can drastically be reduced when the dimensions of the state $n$ and the observation $p$ satisfy $n,p \gg r$ by using the Sherman–Morrison–Woodbury formula. %Because the following equation satisfies.
\begin{align*}
& (C_{U_k}\tilde{R}_{k|k-1} C_{U_k}^\top + M)^{-1} \\
=& M^{-1}
- M^{-1}C_{U_k} 
\\ & \quad \quad \quad \times
\left({\tilde{R}_{k|k-1}}^{-1} + C_{U_k}^\top M^{-1} C_{U_k} \right) ^{-1} C_{U_k}^\top M^{-1} .     
\end{align*}
Since $M^{-1}$ is a time-invariant constant, it can be computed offline and the matrix inverse calculation requires the order of $r^3$, which drastically reduces the computational complexity. 
%to reduce computational complexity.
As in \cite{timecomplexy}, the per-step (iteration) calculation burdens of the Kalman filter (KF), information filter (IF) for reference, and the proposed low-rank Kalman filter (LKF) are summarized in Table \ref{table:per_step_cal}, respectively. 
The proposed filter needs to solve Eq. \eqref{eq:DFOja} at each interval $[t_{k-1},t_{k}]$, we assume that the Euler method is employed to solve \eqref{eq:DFOja} and its iteration number over the interval is denoted by $s$. 
%In Table 1, $s$ is the caluculation step of $U_k$. 

As shown in Table \ref{table:per_step_cal}, the time complexity of KF and IF is $O(n^3)$, whereas the time complexity of LKF is $O(n^2)$. 
Figure \ref{fig:caltime} illustrates the regions of lower time complexity between KF and LKF for $s=4$.  Each curve represents the boundary of the lower time complexity regions for KF and LKF for $p=10,40,100,150$. The upper region of the curve indicates where LKF is more computationally efficient, while the lower region suggests KF is more efficient. From Table \ref{table:per_step_cal}, if we choose $r< 4n/(4+s)$ when $n$ is large, then the proposed filter is computationally superior to the original KF. For linear time-invariant systems, the Oja flow can be precomputed, further reducing the computational burden of the proposed filter.

%For example, when $s=4$, Figure \ref{fig:caltime} shows the calculation burden of KF and LKF. If $n,r$ are on the left side of the curve, the calculation burden of LKF is lower than that of KF. On the right side of the curve, the calculation burden of KF is lower than that of LKF.
%The function of the curve is $\frac{28}{3}r^3+(3n+4sn+7p-\frac{5}{2})r^2+(4n^2+sn^2-n+4pn+5p^2-3p-\frac{5}{6})r-4n^3-(2p-\frac{1-s}{2})n^2-(3p^2+\frac{1-s}{2})n-\frac{8}{3}p^3+p^2-\frac{1}{3}p=0$.} 
%The time complexity of the Low-rank Kalman filter is significantly reduced when $n,p \gg r$.

\begin{table}[!htb]
\caption{Per step calculation burden of filters for linear time-invariant systems}
\label{table:per_step_cal}
\begin{tabular}{|p{2em}|p{18em}|}\hline
          Filter   &   Calculation Burden \\\hline 
   KF & $4n^3+\frac{7}{2}n^2-\frac{3}{2}n+4n^2p+np+3np^2+\frac{1}{6}(16p^3-3p^2-p)$ \\ \hline
  IF & $\frac{1}{6}(50n^3+45n^2-23n)+2n^2p+np$ \\ \hline
  LKF & $(4r+sr+2p+4-\frac{s}{2})n^2+(3r^2+4sr^2-r+4pr+p-2+\frac{s}{2})n+(5r+\frac{1}{2})p^2+(7r^2-3r-\frac{1}{2})p+\frac{1}{6}(56r^3-15r^2-5r)$ \\ \hline
\end{tabular}
\end{table}

\begin{figure}[!htbp]
\centering
\includegraphics[width=1.1\linewidth]{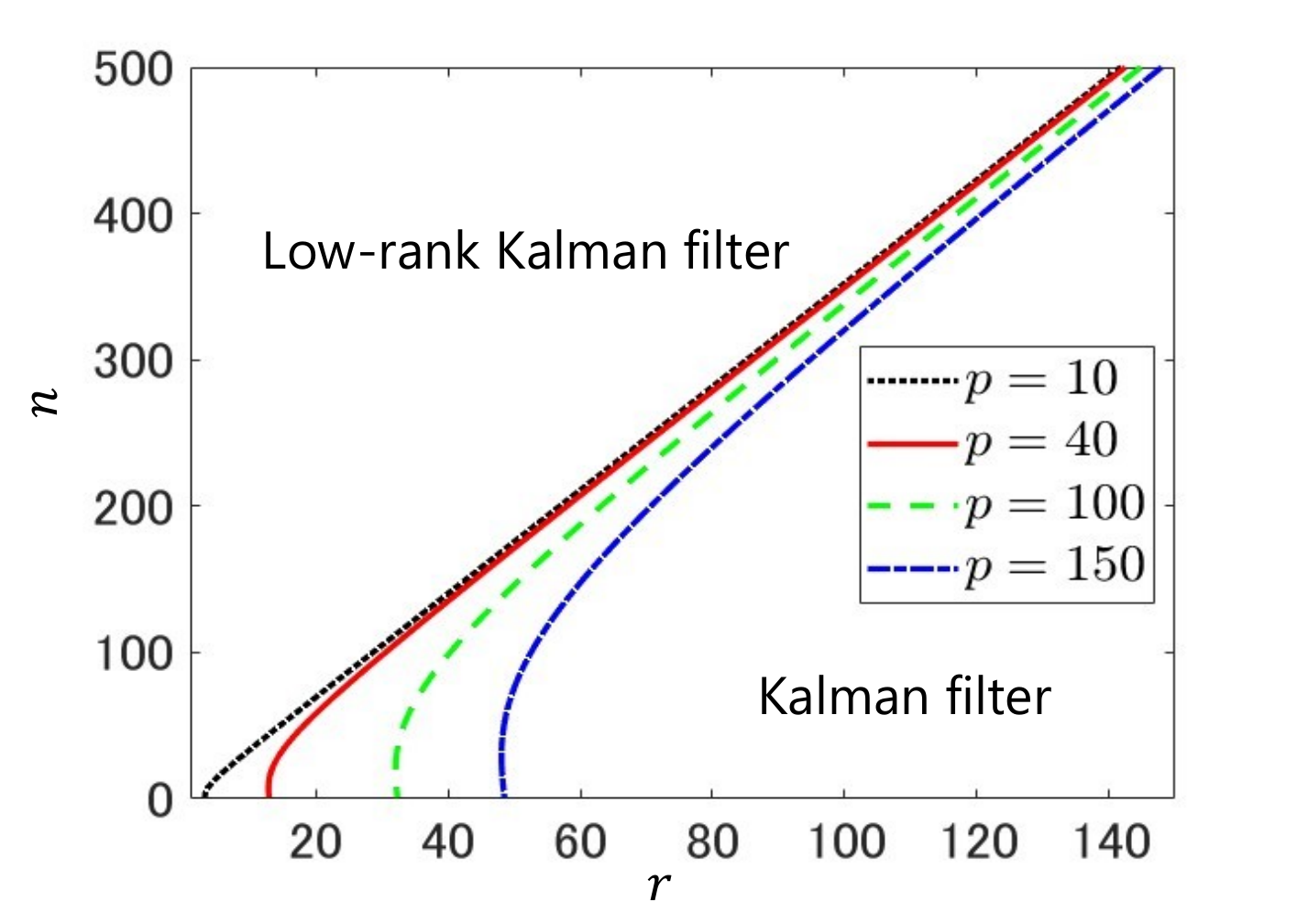}
\caption{The comparison about the time complexity between KF and LKF. The upper region of each curve means that LKF has less time complexity.}
\label{fig:caltime}
\end{figure}

Next, we establish the error covariance matrix $\tilde{V}_{k+1|k} \coloneqq {\mathbb{E}}\left[(x[k+1]-\tilde{x}_{k+1|k})(x[k+1]-\tilde{x}_{k+1|k})^\top\right]$. 
Define the estimation error $\tilde{e}_{k} := x[k] - \tilde{x}_{k|k-1}$.  
Then the time evolution equation of $\tilde{e}_{k}$ is as follows:
\begin{align*}
\tilde{e}_{k+1} = A_d(I_n - U_k\tilde{F}_kC)\tilde{e}_k + G_dw[k] - U_k\tilde{F}_k H v[k].\notag   
\end{align*}
%Here, $\tilde{e}_k = x[k] - \tilde{x}_{k|k-1}$. 
From the above equation, the time evolution equation of $\tilde{V}_{k+1|k}$ is as follows: 
%corresponding to equation (\ref{eq:Ps+1}) for $\tilde{V}_{k+1|k}$.
\begin{align}
\tilde{V}_{k+1|k} &= (A_d - A_dU_k\tilde{F}_kC)\tilde{V}_{k|k-1}(A_d - A_dU_k\tilde{F}_kC)^\top \notag \\
& \quad + G_{d}^{2} + U_k\tilde{F}_kM{\tilde{F}_k}^\top U_{k}^\top. \label{eq:Vs+1}
\end{align}
Similarly to Eq. \eqref{eq:Ps+1}, $A_d(I_n-\bar{U}\tilde{F}C)$ determines the stability of the low-rank Kalman filter.  
%Since Eq. (\ref{eq:Vs+1}) is similar to Eq. (\ref{eq:Ps+1}), the stability of low-rank Kalman filter is equal to that of $A_d(I_n-\bar{U}\tilde{F}C)$. 
Hence, we analyze the stability of $A_d(I_n-\bar{U}\tilde{F}C)$ below.

\subsection{Stability of low-rank Kalman filter}
In this section, we analyze important properties such as the boundedness of estimation errors, and the criterion of rank $r$.
First, we show a few lemmas to analyze those properties. 
Let $\mathcal{U}$ be the asymptotically stable equilibrium set of Eq. \eqref{eq:DFOja}.  
In the reminder of this subsection, we assume that $U_{0}=\bar{U} \in \mathcal{U}$. 
%of stable equilibrium points of $U_k$. Let $\bar{U}$ be an equilibrium point such that $\bar{U} \in \mathcal{U}$. 
%We consider the stable equilibrium points of equation (11). 
For any $\bar{U} \in {\mathcal{U}}$, the following equation holds:
\begin{align}
0=(I_n-\bar{U}{\bar{U}}^{\top})A\bar{U}. \label{eq:equiUs}    
\end{align}

% From Equation (\ref{eq:equiUs}), the following lemma holds.
\begin{lemma}\label{lem:AU_eUTAU} The following equation holds. 
\begin{align*}
%A_{\bar{U}} = 
{\bar{U}}^{\top} A_{d}%e^{Ah} 
\bar{U} = e^{{\bar{U}}^{\top}A\bar{U}h}.    
\end{align*}  
\end{lemma}

\begin{proof}
From Equation \eqref{eq:equiUs}, we have $A^2\bar{U} = A\bar{U}{\bar{U}}^{\top}A\bar{U} = \bar{U} ({\bar{U}}^{\top}A\bar{U})^{2}$. Thus, for any nonnegative integer $m$, ${\bar{U}}^{\top}A^m\bar{U} = ({\bar{U}}^{\top}A\bar{U})^m$ holds. Since $A_{d}=e^{Ah}$, this proves the lemma.  
\end{proof}

Computation of the matrix exponential requires a heavy computational burden if the matrix has a high dimension. Lemma \ref{lem:AU_eUTAU} shows $A_{\bar{U}} = \exp ( {\bar{U}}^{\top}A\bar{U}h )  \in {\mathbb{R}}^{r\times r}$, which can reduce the computational burden if $r\ll n$. 
As of Prop. \ref{prop.eigenvalue_preservation}, the reduced matrix $A_{\bar{U}}$ inherits a part of $A$'s eigenvalues. 

\begin{lemma}\label{lem:AU_eigenvalue}
Let $\lambda _{i} = \lambda_i({\bar{U}}^{\top}A\bar{U})$ be the eigenvalues of ${\bar{U}}^{\top}A\bar{U}$, and $\bm{\psi}_i = \bm{\psi} _{i}({\bar{U}}^{\top}A\bar{U}) \in {\mathbb{C}}^{r}$ be the corresponding eigenvectors. Then, the following holds.
\begin{align*}
{\bar{U}}^{\top}e^{Ah}\bar{U}\bm{\psi}_i = e^{\lambda_i h}\bm{\psi}_i,\quad i=1,\dots ,r.     
\end{align*}  
\end{lemma}

\begin{proof}
Using Lemma \ref{lem:AU_eUTAU} for ${\bar{U}}^{\top}e^{Ah}\bar{U}\bm{\psi}_i$ gives  
\begin{align*}
{\bar{U}}^{\top}e^{Ah}\bar{U} \bm{\psi}_i &= 
\sum _{\ell =0}^{\infty} \frac{(h{\bar{U}}^{\top}A\bar{U})^{\ell}}{\ell !} \bm{\psi}_i \\
%\left(hI_r + h{\bar{U}}^{\top}A\bar{U} + \dots\right) \bm{\psi}_i \\
&= %\sum _{\ell =0}^{\infty} \frac{(h\lambda _{i})^{\ell}}{\ell !} \bm{\psi}_i
\left(1 + h\lambda_i + \dots \right) \bm{\psi}_i 
= e^{\lambda_ih}\bm{\psi}_i.    
\end{align*}
Hence, Lemma \ref{lem:AU_eigenvalue} is proven.    
\end{proof}

Lemma \ref{lem:AU_eigenvalue} and Proposition \ref{prop.eigenvalue_preservation} imply that if $A$ has $r^{\prime} \leq r$ Hurwitz-unstable eigenvalues, then $A_{\bar{U}}$ has the $r^{\prime}$ Schur-unstable eigenvalues.  

\begin{lemma}\label{lem:LKF_observable}
If the system $(C, A) \in {\mathbb{R}}^{p\times n} \times {\mathbb{R}}^{n\times n}$ is observable, then $(C, A_d)\in {\mathbb{R}}^{p\times n} \times {\mathbb{R}}^{n\times n}$ and $(C_{\bar{U}}, A_{\bar{U}})$ $\in {\mathbb{R}}^{p\times r} \times {\mathbb{R}}^{r\times r}$ are also observable.
\end{lemma}

\begin{proof}
It is clear that if the system $(C, A)$ is observable, then the system $(C, A_d)$ is also observable. 
Next, we point out the contradiction to $(C,A)$ observability if $(C_{\bar{U}}, A_{\bar{U}})$ is unobservable. 
If $(C_{\bar{U}}, A_{\bar{U}})$ is unobservable, then there exists an eigenvalue $\lambda_i = \lambda_i (A_{\bar{U}})$ and its corresponding eigenvector ${\bm{v}} _{i}$ such that 
%for eigenvalues $\lambda_i$ of $A_{\bar{U}} = {\bar{U}}^{\top}A\bar{U}$ and corresponding eigenvectors $v_i$, we have the following equation from Lemma \ref{lem:AU_eigenvalue}.
\begin{align*}
A_{\bar{U}} \bm{v}_i = 
{\bar{U}}^{\top}A_d\bar{U}\bm{v}_i = %e^{\lambda_i h} 
\lambda_i \bm{v}_i, \quad C\bar{U}\bm{v}_i = 0.    
\end{align*}
Note that from Prop. \ref{prop.eigenvalue_preservation} and Lemma \ref{lem:AU_eigenvalue}, $\{ \lambda_i (A_{\bar{U}})\}_{i=1}^{r} = \{ e^{\lambda _{i}(A)h} \} _{i=1}^{r} $. 
From Equation (\ref{eq:equiUs}), we also have 
\begin{align*}
\bar{U}A_{\bar{U}} \bm{v}_{i} & = {\bar{U}} {\bar{U}}^{\top}A_d\bar{U}\bm{v}_i = A_d\bar{U}\bm{v}_{i}
= %e^{\lambda _{i} h}
\lambda_i \bar{U}\bm{v}_{i}. %,\quad C\bar{U}\bm{v}_{i} = 0.    
\end{align*}
Therefore, $\bar{U}\bm{v}_{i}$ is an eigenvector of $A_{d}$. 
Since the column vectors of $\bar{U} \in {\mathrm{St}}(r,n)$ are linearly independent, $\bar{U}\bm{v}_{i} \neq 0$. This contradicts the observability of $(C, A_d)$. Hence, If the system $(C, A)$ is observable, then the system $(C, A_d)$ and the system $(C_{\bar{U}}, A_{\bar{U}})$ are also observable.
\end{proof}

\begin{lemma}\label{lem:reachable}
If the system $(A, G)\in {\mathbb{R}}^{n\times n} \times {\mathbb{R}}^{n\times q}$ is reachable, then $(A_d, G_d) \in {\mathbb{R}}^{n\times n} \times {\mathbb{R}}^{n\times n}$ and the system $(A_{\bar{U}}, G_{\bar{U}})\in {\mathbb{R}}^{r\times r} \times {\mathbb{R}}^{r\times n}$ are also reachable.    
\end{lemma}
\begin{proof}
If the system $(A, G)$ is reachable, then the following equation holds. 
\begin{align}
\int^{h}_0e^{A\tau}GG^{\top}e^{A^{\top}\tau}d\tau = G_{d}^{2} >0\notag  
\end{align}
From the above equation, $(A_d, G_d)$ is reachable. Next, we consider 
$\bar{U}^\top G_d^2 \bar{U}$. Since the rank of $\bar{U}^\top G_{d}^{2}\bar{U}$ is $r$, $(A_{\bar{U}}, G_{\bar{U}})$ is reachable. Thus, If the system $(A, G)$ is reachable, then $(A_d, G_d)$ and the system $(A_{\bar{U}}, G_{\bar{U}})$ are also reachable.
\end{proof}

Using Lemmas \ref{lem:LKF_observable} and \ref{lem:reachable}, we can prove a proposition regarding the convergence of $\tilde{R}_{k+1|k}$.

\begin{proposition}\label{prop:tildeR_unique}
 If the system $(A_{\bar{U}}, G_{\bar{U}}, C_{\bar{U}})$ is reachable and observable, then the algebraic Riccati equation:
\begin{align*}
A_{\bar{U}}(\tilde{R}-\tilde{F}C_{\bar{U}}\tilde{R})A_{\bar{U}}^{\top} + G_{\bar{U}} G_{\bar{U}} ^{\top}=\tilde{R}    
\end{align*}
has a unique positive definite solution $\tilde{R}$, and $A_{\bar{U}}(I_r-\tilde{F}C_{\bar{U}})$ is stable. Furthermore, the solution $\tilde{R}_{k+1|k}$ of \eqref{eq:Rs+1} converges to $\tilde{R}$, where $\tilde{F}=\tilde{R}{C_{\bar{U}}}^{\top}(C_{\bar{U}}\tilde{R}{C_{\bar{U}}}^{\top}+M)^{-1}$.   
\end{proposition}

Since $\mathcal{U}$ can contain an infinite number of elements and different $\bar{U}^{\prime} \in \mathcal{U}$ gives different matrix $A_{\bar{U}^{\prime}}$, it is unclear whether the results depend on the choice of the steady-state solution of the Oja flow. 
However, the following holds.  
\begin{proposition}\label{prop:P_uniqueness}
Suppose that $(A_{\bar{U}}, G_{\bar{U}}, C_{\bar{U}})$ is reachable and observable. 
Let $\tilde{R}_{\bar{U}}$ and $\tilde{R}_{\bar{U}^{\prime}}$ be the steady-state solution of Eq. \eqref{eq:Rs+1} with respect to $\bar{U}, \bar{U}^{\prime} \in \mathcal{U}$, respectively.  
Then, $\bar{U} \tilde{R}_{\bar{U}} \bar{U}^{\top} = \bar{U}^{\prime}\tilde{R}_{\bar{U}^{\prime}} \bar{U}^{\prime}{}^{\top}$.  
%(\ref{eq:Vs+1})
\end{proposition}
The proof is the same as in \cite[Proposition 4]{yamada2021comparison}, so we skip the proof.  
Therefore, the error covariance \eqref{eq:Vs+1} is uniquely determined. 

From Proposition \ref{prop:tildeR_unique}, the following proposition regarding the stability of the matrix $A_d(I_n-\bar{U}\tilde{F}C)$ holds.
\begin{proposition}\label{prop:LKF_eigenvalue}
Let $\sigma_1, \dots, \sigma_r$ be the eigenvalues of the matrix $A_{\bar{U}}(I_r-\tilde{F}C_{\bar{U}}) \in {\mathbb{R}}^{r\times r}$. Then, the eigenvalues of the matrix $A_d(I_n-\bar{U}\tilde{F}C) \in {\mathbb{R}}^{n\times n}$ are $\sigma_1, \dots, \sigma_r$, $\exp(\lambda_{r+1}(A) h), \dots, \exp(\lambda_n(A)h)$.    
\end{proposition}
Notice that $\exp(\lambda _{i}(A)h)$, $i=1,\dots ,n$, are the eigenvalues of $A_{d}$ and $\bm{\psi}_{i}(A)$ are the corresponding eigenvectors of $A_{d}$ and also $A$.  From the definition, $\bm{\psi}_{i}(A_d)$ is not the eigenvector or generalized eigenvector of $\exp(\lambda _{i}(A)h)$ in general. 

\begin{proof}
To prove the statements, we first show the following.
\begin{align}
& (I_{n} - \bar{U}\bar{U}^{\top}) A_d(I_n-\bar{U}\tilde{F}C) \bar{U} \notag \\
=& (I_{n} - \bar{U}\bar{U}^{\top}) A_d \bar{U} (I_{r} - \tilde{F}C \bar{U})
%- (I_{n} - \bar{U}\bar{U}^{\top}) A_d \bar{U}\tilde{F}C\bar{U}
=
O_{n,r} \label{eq:supplemental_equality}
\end{align}
The equalities follow Eq. \eqref{eq:equiUs} and Lemma \ref{lem:AU_eUTAU}.  
The above relation means that choosing a suitable coordinate, $A_d(I_n-\bar{U}\tilde{F}C)$ can be block diagonalized, and $\bar{U}\bar{U}^{\top} A_d(I_n-\bar{U}\tilde{F}C) \bar{U} \bar{U}^{\top}$ is the left-upper block while $(I_n-\bar{U}{\bar{U}}^{\top})A_d(I_n-\bar{U}\tilde{F}C) (I_n-\bar{U}{\bar{U}}^{\top})$ is the right-lower block. For the proof, it is enough to investigate the eigenvalues of these two blocks of the matrix.

Let $\bm{u}_i  = {\bm{\psi}}_{i}(A_{\bar{U}}(I_r-\tilde{F}C_{\bar{U}})) \in {\mathbb{C}}^{r}$, $i=1,\dots ,r$.  
Hereafter, we only use eigenvectors to investigate the eigenvalues if the eigenvalues are degenerated. 
%, be the eigenvectors corresponding to the eigenvalues $\sigma_1, \dots, \sigma_r$ of the matrix $A_{\bar{U}}(I_r-\tilde{F}C_{\bar{U}}) {\color{blue}\in {\mathbb{C}}^{r\times r}}$. 
Then, using Eq. \eqref{eq:supplemental_equality}, the following equation holds.
\begin{align*}
\bar{U}A_{\bar{U}}(I_r-\tilde{F}C_{\bar{U}})\bm{u}_i =& A_d(I_n-\bar{U}\tilde{F}C)\bar{U}\bm{u}_i 
= \sigma_i\bar{U}\bm{u}_i.    
\end{align*}
Since $\bar{U}\bm{u}_i \neq 0$, $\sigma_1, \dots, \sigma_r$ are eigenvalues of $A_d(I_n-\bar{U}\tilde{F}C)$. 
Next, for $m > r$, $(I_n - \bar{U} {\bar{U}}^{\top}) {\bm{\psi}}_m(A) \neq 0$ holds because each row vector $\bar{U}$ is linearly dependent of ${\bm{\psi}}_i (A)$ for $i=1,\dots ,r$ and ${\bm{\psi}} _{m}(A)$, $m>r$, is independent from ${\bm{\psi}}_i (A)$. 
Hence, from Eq. \eqref{eq:equiUs},
\begin{align*}
\bm{\psi}^{\dagger}_m(A)(I_n-\bar{U}{\bar{U}}^{\top}) (A_d-A_d\bar{U}\tilde{F}C-e^{\lambda _{m}(A)h} I_n )\\
\times(I_n-\bar{U}{\bar{U}}^{\top}) \bm{\psi}_m(A) = 0.    
\end{align*}
Thus, $\exp(\lambda_{r+1}(A) h), \dots, \exp(\lambda_n(A)h)$ are also eigenvalues of $A_d(I_n-\bar{U}\tilde{F}C)$. Therefore, the eigenvalues of $A_d(I_n-\bar{U}\tilde{F}C)$ are $\sigma_1, \dots, \sigma_r$, $e^{\lambda_{r+1}(A) h}, \dots, e^{\lambda_n(A)h}$.   
\end{proof}

Note that from Prop. \ref{prop:P_uniqueness}, $A_d(I_n-\bar{U}\tilde{F}C)$ is uniquely determined.  
From Proposition \ref{prop:tildeR_unique}, if we choose the rank $r$ such that ${\mathrm{Re}}(\lambda_{r+1}(A)) < 0$, the proposed low rank Kalman filter is stable. Additionally, it is equal to guarantee the boundedness of the error covariance matrix $\tilde{V}_{k+1|k}$.
Next, we establish how to choose the rank $r$.
\begin{theorem}\label{therem:stabilitycondtion_rank}
%Suppose a system $(A_d, G_d, C)$ that is reachable and observable. 
Suppose system $(A_d, G_d, C)$ is reachable and observable, and $A$ has $r^{\prime}$ Hurwitz unstable eigenvalues. Then, $A_d(I_n-\bar{U}\tilde{F}C)$ is Schur stable if and only if $r\geq r^{\prime}$.
 %If $A$ has $r'$ Hurwitz unstable eigenvalues and $A_d(I_n-\bar{U}\tilde{F}C)$ is Schur stable, then $r\ge r'$.
\end{theorem}
\begin{proof}
From Proposition \ref{prop:KF stability} and Lemmas \ref{lem:LKF_observable} and \ref{lem:reachable}, $\sigma_1, \dots, \sigma_r$ are stable eigenvalues. Thus, from Prop. \ref{prop:LKF_eigenvalue}, if $A_d(I_n-\bar{U}\tilde{F}C)$ is stable, then $r\le r^{\prime}$. 
Next,if $r<r^{\prime}$, then, from Proposition \ref{prop:KF stability}, $A_d(I_n-\bar{U}\tilde{F}C)$ includes an eigenvalue $e^{\lambda _{r^\prime}(A) h} $ outside the unit disc; this contradicts to the assumption of stability. Therefore, $r\geq r^{\prime}$.
%$\lambda_{r^{\prime}}(A_d(I_n-\bar{U}\tilde{F}C))=\lambda_{r^{\prime}}(A_d)$,  which means $A_d(I_n-\bar{U}\tilde{F}C)$ has an unstable eigenvalue; this contradicts to the assumption. Therefore, $r\ge r^{\prime}$.
%From Proposition \ref{prop:LKF_eigenvalue}, if $r<r'$, then $\lambda_{r'}(A_d(I_n-\bar{U}\tilde{F}C))=\lambda_{r'}(A_d)$,  which means $A_d(I_n-\bar{U}\tilde{F}C)$ has an unstable eigenvalue; this contradicts to the assumption. Therefore, $r\ge r'$. 
\end{proof}

This theorem provides a design criterion for $r$ to guarantee the boundedness of estimation errors.
If the characteristic polynomial of the system is given, the Jury's stability criterion can count the number of roots outside the unit disc \cite[Thm. 3.3]{astrom1984computer}. 

\section{numerical simulation}
In this section, we demonstrate two numerical examples of the proposed low-rank Kalman filter.  
%to verify the mean square error{\color{blue} of the one-step ahead estimation ${\mathrm{Tr}} [ \tilde{V}_{k+1|k} ]$ of the proposed low-rank Kalman filter. 
%An example in Section 4.1 shows the validity of Theorem 1 and another example in Section 4.2 shows our proposed filter also works under the detectability condition. 
%The last example of Section 4.3 illustrates the rank $r$ dependency of the steady state estimation errors.  

% Section 4.1 and 4.3 satisfy the assumption of Thm. 1, while Sec. 4.2 doesn't satisfy them.}

\subsection{Verification of the bounded estimations}

    Consider the following system parameters: $n=10$, $A$ is randomly generated, $C=\left[I_{4},\ O_{4,6}\right]$, $G=I_{10}$, $H = I_{4}$, $h=0.01$, $r=6$, $\Sigma_0=I_{10}$, $\tilde{\Sigma}_0=I_6$. 
    The generated $A$ has $6$ Hurwitz unstable eigenvalues and the system is observable and reachable. 
    Set $U_0$ as $[I_6 , O_{6,4}]^{\top}$ and use $\varepsilon=0.01$ for Eq. \eqref{eq:DFOja}. 
    From the statements of Theorem \ref{therem:stabilitycondtion_rank}, the minimum $r$ is 6 and Figure \ref{fig1:n=6} shows the mean square errors of the low-rank Kalman filter.  
    As seen, the estimation error is bounded.  
    %We also checked the case $r=5$, but the mean square estimation error diverges in this case.  
    
    % Here, we show two examples to illustrate the estimation errors of the proposed filter under the following system parameters. 
    % \begin{description}
    % \item[\upshape Case 1 (observable and reachable):] $n=10$, $A$ is randomly generated, $C=\left[I_{4},\ O_{4,6}\right]$, $G=I_{10}$, $H = I_{4}$, $h=0.01$, $r=6$, $\Sigma_0=I_{10}$, $\tilde{\Sigma}_0=I_6$.

    % \item[\upshape Case 2 (detectable and reachable):] $n=100$, $A$ is randomly generated, $C=\left[I_{40},\ O_{40,60}\right]$, $G=I_{100}$, $H=I_{40}$, $h=0.01$, $r=51,56,61$, $\Sigma_0=I_{100}$, $\tilde{\Sigma}_0=I_r$.
    
    % \end{description}
    % The number of unstable eigenvalues in Case 1 is 6, while 51 in Case 2. The $A$ matrix in Case 2 was adopted 

% \subsection{Simulation 1}
% Let $A$ be a randomly generated matrix with $n=10$. In this case, $A$ has $6$ Hurwitz unstable eigenvalues.
% Then, we consider the following system parameters:
% \begin{align*}
% C=\left[I_4,\ O_{4,6}\right],\ G=I_{10},\ H=I_4,\ h=0.01,\\
% \ r=6, \ \Sigma_0=I_{10},\ \tilde{\Sigma}_0=I_6.
% \end{align*}
% This system is reachable and observable. 
% Set $U_0$ as $[I_r , O_{6,4}]^{\top}$. {\color{blue}$\epsilon=0.01$}. 
% Then, the mean square errors at each time calculated by Kalman filter (KF), and low-rank Kalman filter (LKF) proposed are shown in Fig. \ref{fig1:n=6}.

\begin{figure}[!htbp]
\centering
\includegraphics[width=1.0\linewidth]{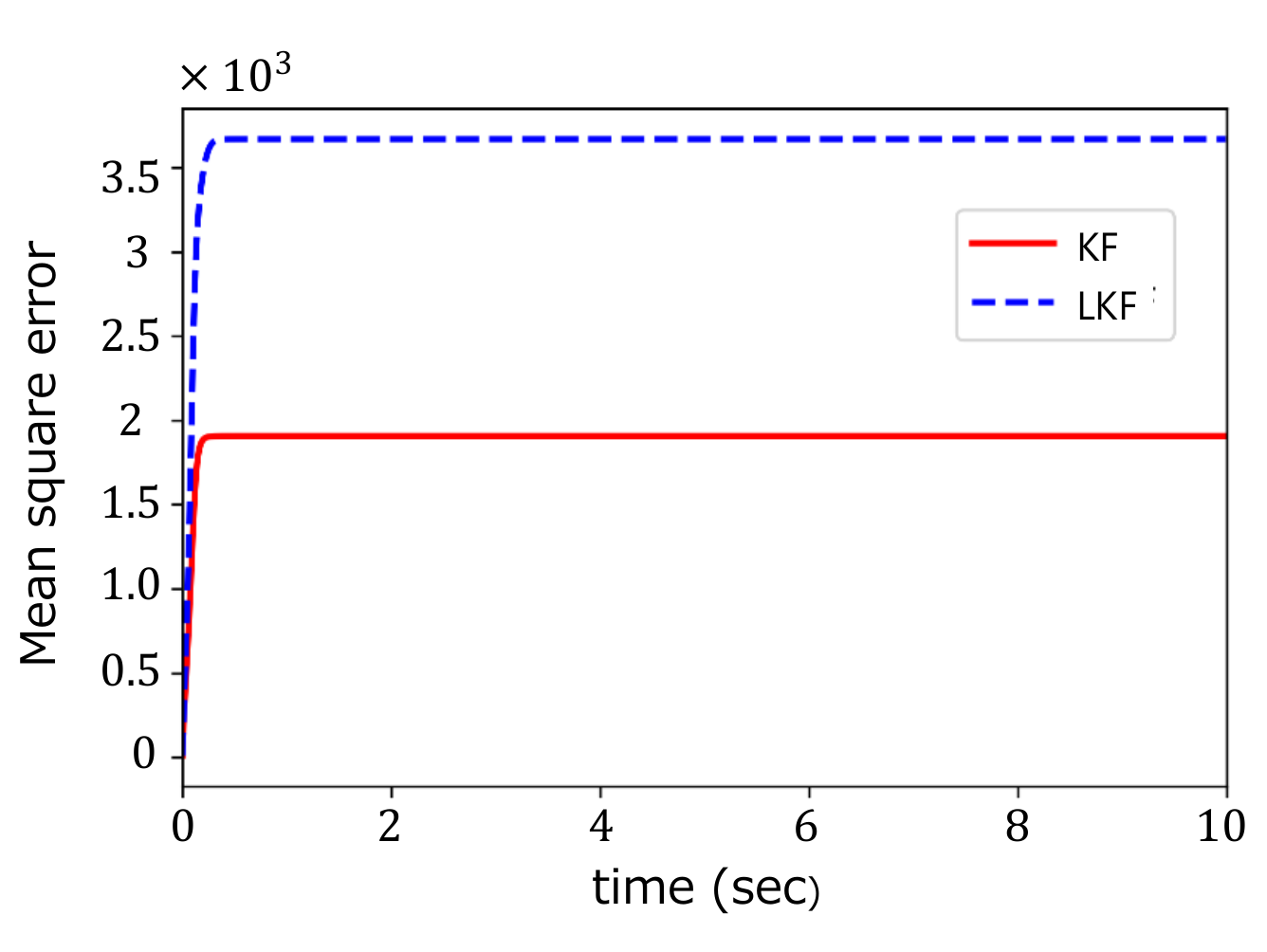}
\caption{Time evolution of the mean square one-step ahead estimation errors ${\mathrm{Tr}} [ \tilde{V}_{k+1|k} ]$}
\label{fig1:n=6}
\end{figure}

%From Fig. \ref{fig1:n=6}, we confirm Theorem 1 holds.  

% \subsection{Simulation 2}
% Let $A$ be a randomly generated matrix with $n=100$. 
% %Let $n=100$. $A$ is the matrix we randomly generated. 
% In this case, $A$ has $51$ unstable eigenvalues. 
% Then, we consider the following system:
% \begin{align*}
% C=\left[I_{40},\ O_{40,60}\right],\ G=I_{100},\ H =I_{40},\ h=0.01,\\
% \ \Sigma_0=I_{100},\ \tilde{\Sigma}_0=I_r.
% \end{align*}
% $U_0$ is $U_0\in{\mathcal{U}}$. This system is unreachable and unobservable but detectable.
% Then, {\color{blue}the dependency of the rank $r$ of LKF is shown in Fig. \ref{fig2:n=100}.}
% %the relationship {\color{blue}between} the mean square errors {\color{blue}of LKF} at each time calculated by low-rank Kalman and rank $r$ is shown in Fig. \ref{fig2:n=100}. 

% \begin{figure}[!h]
% \centering
% \includegraphics[width=1.1\linewidth]{SICE2.pdf}
% \caption{Comparison of ${\mathrm{Tr}} [ \tilde{V}_{k+1|k} ]$ under different $r$}
% \label{fig2:n=100}
% \end{figure}

% {\color{blue}Figure \ref{fig2:n=100} shows that larger rank $r$ gives better estimation accuracy.} 
% %From Fig. \ref{fig2:n=100}, it is confirmed that selecting the rank $r$ somewhat larger than the number of unstable eigenvalues leads to better estimation accuracy. 
%  Further investigation is necessary for a precise estimation of a small $r$.  

 \subsection{Impact of the choice of $r$}
    Next, we examine how the choice of $r$ impacts the estimation errors.  
    Consider the following system parameters: $n=100$, $A$ is a randomly generated symmetric matrix, $C=\left[I_{40},\ O_{40,60}\right]$, $G=I_{100}$, $H = I_{40}$, $h=0.001$,  $\Sigma_0=I_{100}$, $\tilde{\Sigma}_0=I_r$. 
    The generated $A$ has $49$ Hurwitz unstable eigenvalues and the system is observable and reachable. 
    Then, we examine relationship rank $r$ and steady mean square error of one-step estimation errors ${\mathrm{Tr}}[ \tilde{V}_{\infty}]$, where $\tilde{V}_{\infty} := \lim_{k\to\infty}\tilde{V}_{k+1|k}$ in Equation \eqref{eq:Vs+1}. 
    The dependency of $r$ for the steady state estimation errors, normalized by that of KF, is shown in Figure \ref{fig3:r vs MSE}.

% If $A$ is a symmetric matrix, then for any $r$, there is $\mathcal{U}$. Therefore, we consider a symmetric matrix for ease of comparison. 
% Let $A$ be a randomly generated symmetric matrix with $n=100$. 

\begin{figure}[!h]
\centering
\includegraphics[width=1.0\linewidth]{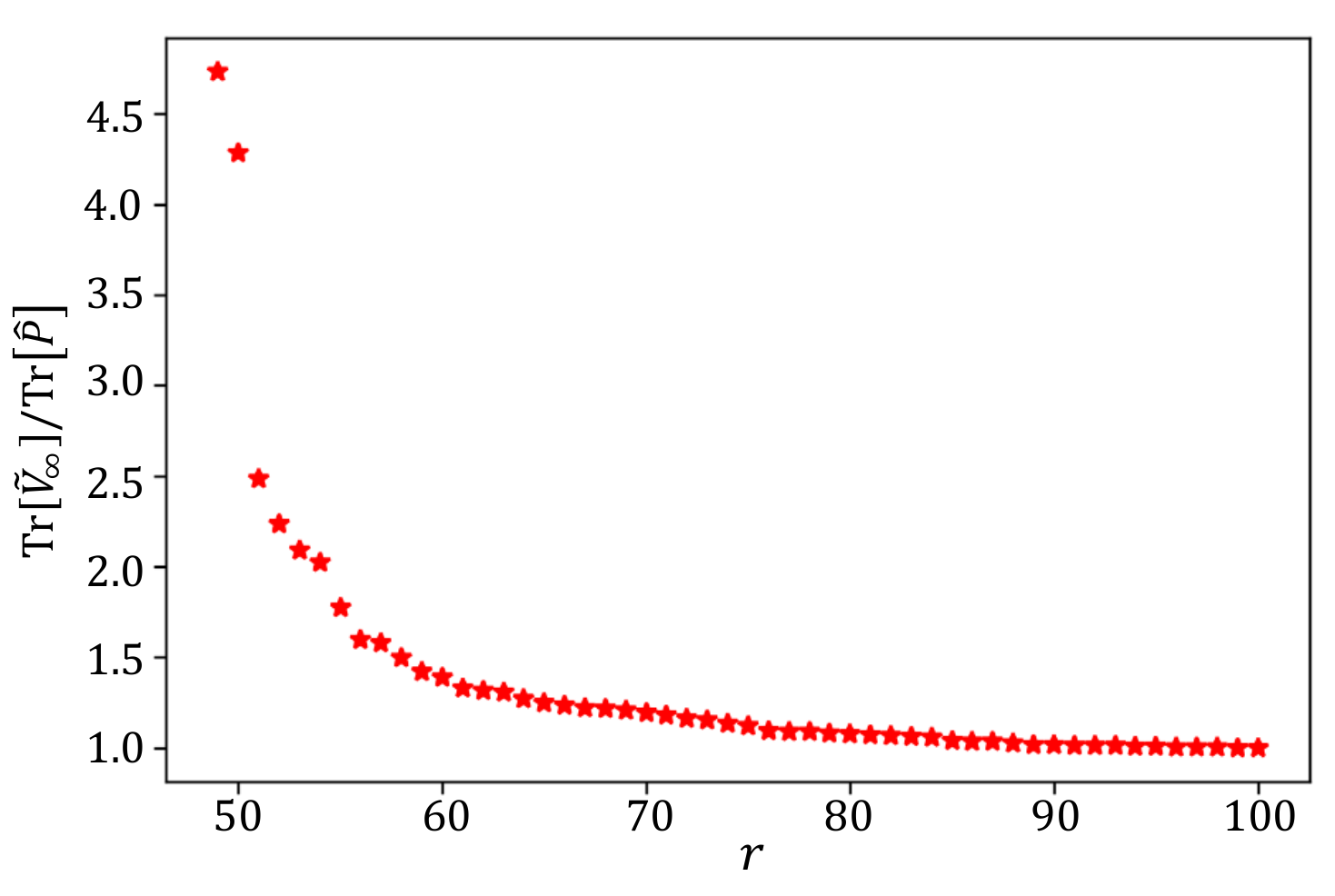}
\caption{${\mathrm{Tr}} [ \tilde{V}_{\infty} ]/\mathrm{Tr}[\hat{P}]$ versus $r$}
\label{fig3:r vs MSE}
\end{figure}

% In this case, $A$ has $49$ unstable eigenvalues. 
% Then, we consider the following system: $
% C=\left[I_{40},\ O_{40,60}\right]$, $G=I_{100}$, $H = I_{40}$, $h=0.001$, $\tilde{\Sigma}_0=I_r$. 
% $U_0$ is $U_0\in{\mathcal{U}}$. This system is reachable and observable.

From Fig. \ref{fig3:r vs MSE}, if $r$ increases, ${\mathrm{Tr}}[ \tilde{V}_{\infty}]$ monotonically decreases. Choosing $r=49$, the minimum requirement for bounded estimation errors, is the worst, but in this example, we can see that slightly increasing $r$ drastically decreases ${\mathrm{Tr}}[ \tilde{V}_{\infty}]$. 
%Especially when $r\in \{49,\dots ,60\}$, the mean square error is significantly reduced.   

\section{conclusion}

In this paper, we proposed a new low-rank Kalman filter to reduce computational complexity while maintaining bounded estimation errors. We analyzed the properties of this filter, including the boundedness of estimation errors and the calculation burden, and identified the rank $r$ condition necessary for filter stability.

In this paper, we focused on continuous-time systems with discrete-time observations. In the future, we will propose and analyze a low-rank Kalman filter for general discrete-time LTI, linear time-varying, and nonlinear stochastic systems.  

While revising this paper, we realized the seminal works on continuous-time reduced QR-decomposition-based low-rank filters for continuous-time linear time-varying and nonlinear systems \cite{trevisan2011kalman,frank2018detectability,tranninger2020uniform,tranninger2022detectability}. The continuous-time reduced QR-decomposition algorithm is essentially the same as the Oja flow \eqref{eq:DFOja}; the solution of the reduced continuous-time QR-decomposition algorithm is described as $U^{\prime}(t) = U(t) W(t) \in {\mathrm{St}}(r,n)$, where $U(t) \in {\mathrm{St}}(r,n)$ is the solution of the Oja flow and the $W(t) \in {\mathbb{R}}^{r\times r}$ is a orthogonal matrix that is the solution of $\frac{d}{dt}W(t) = W(t) S(t)$ with a specific skew-symmetric matrix $S(t)=-S(t)^{\top} \in {\mathbb{R}}^{r\times r}$. The resulting reduced filters are essentially the same as of \cite{yamada2021new}.  

It is noteworthy that $U^{\prime}(t)$ converges to the invariant set $\mathcal{U}$ but never converges to any point of $\mathcal{U}$. Consequently, the reduced Riccati equation using $U^{\prime}(t)$ lacks equilibrium points. Therefore, our Oja flow-based approach remains meaningful for the steady-state low-rank Kalman filter.

\section{Acknowledgement}

This work was supported by JSPS KAKENHI Grant Numbers JP19K03619 and JP23K26126. 

%%%%%%%%%%%%%%%% BIBLIOGRAPHY IN THE LaTeX file !!!!! %%%%%%%%%%%%%%%%%%%%%%
%\bibliographystyle{ieeetr}
%\bibliography{reference.bib}

\end{document}